\documentclass[a4paper,12pt]{article}

\usepackage{amsmath,amssymb,amsthm,latexsym,url,multirow,color,booktabs,slashed,verbatim,mathrsfs,bbm,enumitem,tabularx,graphics}
\usepackage{mathrsfs}
\numberwithin{equation}{section}

\def\vect#1{\mbox{\boldmath $#1$}}

\def\s{\sigma}
\def\e{\varepsilon}

\def\a{\alpha}

\def\vp{\varphi}
\def\l{\lambda}
\def\D{\mathscr{D}}%\def\D{\slashed{\mathcal{D}}}
\def\V{\mathcal{V}}
\def\S{\mathcal{S}}

\def\H{\mathfrak{H}}
\def\K{\mathscr{K}}
\def\Cyl{\mathcal{C}}
\def\F{\mathscr{F}}
\def\chrt{\mathbbm{1}}
\def\C{\mathbb{C}}
\def\R{\mathbb{R}}

\def\N{\mathbb{N}}

\def\wh{\widehat}

\def\o{\overline}
\def\til{~}
\newcommand{\n}[1]{\left\lVert#1\right\rVert}
\DeclareMathOperator*{\sgn}{sgn}
\DeclareMathOperator*{\supp}{supp}
\DeclareMathOperator*{\dom}{dom}
\DeclareMathOperator*{\ran}{ran}
\newtheorem{theorem}{Theorem}

\newtheorem{remark}{Remark}[section]
\newtheorem{lemma}{Lemma}

\newtheorem*{hypothesis}{Hypothesis}
\newtheorem*{notations}{Notations}

\newlength{\longtext}
\title{Eigenvalue bounds for non-selfadjoint Dirac operators}
\author{
Piero D'Ancona,
\quad
Luca Fanelli,
\quad
Nico Michele Schiavone
\thanks{\hrule \vspace{5pt}
	Department of Mathematics ``Guido Castelnuovo'', University of Rome ``La Sapienza'',
	Piazzale Aldo Moro 5, 00185 Rome, Italy.
	\newline
	e-mails: dancona@mat.uniroma1.it, fanelli@mat.uniroma1.it, schiavone@mat.uniroma1.it}
}

\date{
	\footnotesize
	\begin{tabularx}{0.7\textwidth}{lX}
		\begin{tabular}{@{}l@{}}
		\end{tabular}
		& 
		\begin{tabular}{@{}l@{}}
			{Dirac operator,} 
			{non-selfadjoint perturbation,}
			\\
			{localization of eigenvalues,}
			{Birman–Schwinger principle} 
		\end{tabular}
		\\[0.35cm]
		{\bf MSC2020:}
		&
		primary 35P15, 35J99, 47A10, 47F05, 81Q12
	\end{tabularx}
}

\renewcommand\footnotemark{}

\begin{document}
\maketitle

%%%%%

\begin{abstract}
	In this work we prove that the eigenvalues of the $n$-dimensional massive Dirac operator $\D_0 + V$, $n\ge2$, perturbed by a possibly non-Hermitian potential $V$, are localized in the union of two disjoint disks of the complex plane, provided that $V$ is sufficiently small with respect to the mixed norms $L^1_{x_j} L^\infty_{\widehat{x}_j}$, for $j\in\{1,\dots,n\}$. In the massless case, we prove instead that the discrete spectrum is empty under the same smallness assumption on $V$, and in particular the spectrum is the same of the unperturbed operator, namely $\s(\D_0+V)=\s(\D_0)=\R$.
	
	The main tools we employ are an abstract version of the Birman-Schwinger principle, which include also the study of embedded eigenvalues, and suitable resolvent estimates for the Schr\"odinger operator.
\end{abstract}

%%%%%%%%%%%%%%%%%%%%%%%%%%%%%%%%%%%%%%%%%%%%%%%%%%
%%%%%%%%%%%%%%%%%%%%% SECTION1 %%%%%%%%%%%%%%%%%%%%%%%
%%%%%%%%%%%%%%%%%%%%%%%%%%%%%%%%%%%%%%%%%%%%%%%%%%

\section{Introduction}
In recent years, non-selfadjoint operators are attracting increasing attention, not only in view of applications in quantum mechanics and other branches of physics, but also for the interesting mathematical challenges they present. 
While the theory of selfadjoint operators is consolidated in a vast literature, references on the study of non-selfadjoint operators are more sparse, so much that, quoting E.B. Davies \cite{Dav07}, \textit{\lq\lq it can hardly be called a theory''}. We refer to the books \cite{BGSZ15, Tre08} for milestones on the history of the argument, and to \cite{Dav02} for some physical applications.

%When we approach the non-selfadjoint operators, most of the main tools of the selfadjoint, such as the spectral theorem and the variational methods no longer apply. Thus, we are forced to research other techniques, as the Birman-Schwinger principle, which in this work will be our principal tools together with a suitable resolvent estimate.

In this paper, we deal with the free Dirac operator $\D_0$ perturbed by a potential, formally defined by
\begin{equation*}\label{DV}
	\D_{V} = \D_{0} + V.
\end{equation*}
The relevance of these kind of operators in quantum physics is common knowledge, since in the $2$-dimensional case the operator $\D_{V}$ is related to the quantum theory of graphene, while in the $3$-dimensional case the Hamiltonian $\D_{V}$ determines the dynamic of a relativistic quantum particle of spin $\frac{1}{2}$, subject to an external electric field described by the potential $V$.

We consider the operator $\D_{V}$ acting on the Hilbert space of spinors $\H = L^2(\R^n; \C^{N} )$, where $n\ge2$, $N := 2^{\lceil n/2 \rceil}$ and ${\lceil \cdot \rceil}$ is the ceiling function. 
The free Dirac operator $\D_{0}$ with non-negative mass $m$ is defined as
\begin{equation*}
	\D_{0} = -i c \hbar \, \vect{\a} \cdot \nabla + m c^2 \a_0
		 = -i c \hbar \sum_{k=1}^{n} \a_k \frac{\partial}{\partial x_k} + m c^2 \a_0,
	%\v{\a}=(\a_1,\a_2,\a_3),
	%	\quad
\end{equation*}
where $c$ is the speed of light, $\hbar$ is the reduced Planck constant and the matrices $\a_k \in \C^{N \times N}$, for $k \in \{0,\dots,n\}$, are elements of the Clifford algebra (see \cite{Obo99}) satisfying the anti-commutation relations 
\begin{equation}\label{clifford}
	\a_j \a_k + \a_k \a_j = 2\delta_{j,k} I_{\C^N}, \quad\text{for $j,k \in \{0,\dots,n\}$,}
\end{equation}
where $\delta_{j,k}$ is the Kronecker symbol. Without loss of generality we can take
\begin{equation*}
\a_0 =
\begin{pmatrix}
I_{\C^{N/2 \times N/2}} & \vect{0}\\
\vect{0} & -I_{\C^{N/2 \times N/2}}
\end{pmatrix}
\end{equation*}
and renormalize the unit measures such that $c=\hbar=1$.
The free Dirac operator has domain
\begin{equation*}
	\dom(\D_0) = \{ \psi \in \H \colon \nabla\psi \in \H^n \}
\end{equation*}
and it is selfadjoint with core $C_0^\infty(\R^n; \C^N)$.
%
\begin{comment}
In particular, for the $3$-dimensional case, they are defined as

\begin{equation*}
\a_0 =
\begin{pmatrix}
I_2 & \v{0}\\
\v{0} & -I_2
\end{pmatrix},
\quad
\a_i=
\begin{pmatrix}
\v{0} & \s_i \\
\s_i & \v{0}
\end{pmatrix},
\end{equation*}
where $\s_i$, for $i=1,2,3$, are the Pauli matrices
\begin{equation*}
\s_1=
\begin{pmatrix}
0 & 1 \\
1 & 0
\end{pmatrix},
\quad
\s_2 =
\begin{pmatrix}
0 & -i \\
i & 0
\end{pmatrix},
\quad
\s_3 =
\begin{pmatrix}
1 & 0 \\
0 & -1
\end{pmatrix}.
\end{equation*}
\end{comment}

The potential $V \colon \R^n \to \C^{N \times N}$ is allowed to be any non-Hermitian matrix-valued function with $|V|\in L_{\text{loc}}^2(\R^n;\R)$, where $|V|$ is the operator norm of $V$.
Thus the resulting operator may be non-selfadjoint.
With an abuse of notation, we use the same symbol $V$ to indicate the multiplication operator generated by the matrix $V$ in $\H$ with initial domain $\dom(V)=C_0^\infty(\R^n;\C^N)$.

In this work, we are interested in location of eigenvalues for $\D_V$ in the complex plane. For the non-selfadjoint Schr\"odinger operator $-\Delta+V$, we refer to the works by Frank \cite{Fra11,Fra18}, Frank and Sabin \cite{FSab17}, Frank and Simon \cite{FSim15}. In particular, we have that the eigenvalues of $-\Delta+V$ satisfy the bound
\begin{equation*}
	|z|^{\gamma} \le D_{\gamma,n} \int_{\R_n} |V(x)|^{\gamma+n/2}dx,
	\qquad
	0 < \gamma
	\begin{cases}
	=\frac{1}{2}, &\text{if $n=1$,}
	\\
	\le \frac{1}{2}, &\text{if $n\ge2$,}
	\end{cases}
\end{equation*}
where the constant $D_{\gamma,n}>0$ is independent of $z$ and $V$.
This localization estimate was proven by Abramov, Aslanyan and Davies \cite{AAD01} for $n=1$ with the sharp constant $D_{1/2,1}=1/2$, and for larger $n$ by Frank \cite{Fra11}, combining the Birman-Schwinger principle with the  uniform Sobolev inequalities by Kenig, Ruiz and Sogge \cite{KRS87}, i.e.
\begin{equation*}
	\n{(-\Delta-z)^{-1}}_{L^p \to L^{p'}}
	\le C
	|z|^{-n/2+n/p-1},
	\qquad
	\frac{2}{n+1} \le \frac{1}{p}-\frac{1}{p'} \le \frac{2}{n},
	%\frac{2n}{n+2} < p \le \frac{2(n+1)}{n+3}.	
\end{equation*}
where $p'= p/(p-1)$ is the dual exponent of $p$.
In \cite{LS09,Saf10}, Laptev and Safronov conjectured that the range of $\gamma$ for $n\ge2$ can be extended to $0<\gamma<n/2$, and Frank and Simon \cite{FSim15} proved the conjecture to be true for radial symmetric potentials.

Let us return to the Dirac operator $\D_{V}$. If we suppose $V\colon \R^n\to\C^{N \times N}$ is an Hermitian matrix-valued function, such that the operator $\D_{V}$ is selfadjoint, we have an extensive literature on its spectral properties, see for example the monograph by Thaller \cite{Tha13}. In the non-selfadjoint case, the study of the spectrum of $\D_{V}$ was initiated by Cuenin, Laptev and Tretter in \cite{CLT14} in the $1$-dimensional case, followed by 
\cite{Cue14,CS18,Enb18}.
%Cuenin \cite{Cue14}, Cuenin and Siegl \cite{CS18} and Enblom \cite{Enb18}. 
For the higher dimensional case, we refer to the works 
\cite{Cue17,Dub14,FK19,Sam16}.
%by Cuenin \cite{Cue17}, Dubuisson \cite{Dub14}, Fanelli and Krej\v ci\v r\' ik \cite{FK19} and Sambou \cite{Sam16}.

In \cite{CLT14}, the authors proved that, for $n=1$, if $V=(V_{ij})_{i,j \in \{1,2\}}$ with $V_{ij} \in L^1(\R)$ and 
\begin{equation*}
	\n{V}_{L^1(\R)} = \int_{\R} |V(x)|dx \le 1, 	
\end{equation*}
%where $|V(x)|$ is the operator norm of $V(x)$,
then every non-embedded eigenvalue $z\in \rho(\D_{0})$ of $\D_{V}$ lies in the disjoint union
\begin{equation*}
	z \in \overline{B}_{R_0}(x^-_0) \cup \overline{B}_{R_0}(x^+_0)
\end{equation*}
of the two closed disks in the complex plane with centers and radius respectively
\begin{equation*}
	x^\pm_0 = \pm \sqrt{\frac{\n{V}_1^4-2\n{V}_1^2+2}{4(1-\n{V}_1^2)} + \frac{1}{2}},
	\qquad
	R_0 =
	\sqrt{\frac{\n{V}_1^4-2\n{V}_1^2+2}{4(1-\n{V}_1^2)} - \frac{1}{2}}.
\end{equation*}
In particular, in the massless case ($m=0$), the spectrum of $\D_{V}$ is $\R$. They also showed that this inclusion is sharp.
Their proof is essentially based on the combination of the Birman-Schwinger principle with the resolvent estimate for the free Dirac operator
\begin{equation*}
	\n{(\D_{0}-z)^{-1}}_{L^\infty(\R) \to L^1(\R)}
	\le
	\sqrt{\frac{1}{2}
		+\frac{1}{4} \left\lvert\frac{z+m}{z-m}\right\rvert
		+\frac{1}{4} \left\lvert\frac{z-m}{z+m}\right\rvert
	},
	\quad
	\text{$z \in \rho(\D_{0})$}.
\end{equation*}
It should be remarked that, in higher dimensions $n\ge2$, $L^p(\R^n) \to L^{p'}(\R^n)$ estimates for $(\D_0-z)^{-1}$ do not exist, as observed in the Introduction of \cite{Cue14}. Indeed, Cuenin points out that, due to the Stein-Thomas restriction theorem and standard estimates for Bessel potentials, the resolvent $(\D_{0}-z)^{-1} \colon L^p(\R^n) \to L^{p'}(\R^n)$ is bounded uniformly in $|z|>1$ if and only if
\begin{equation*}
	\frac{2}{n+1}
	\le
	\frac{1}{p} + \frac{1}{p'}
	\le
	\frac{1}{n},
\end{equation*}
and thus we are forced to choose $n=1$. The situation for the Schr\"odinger operator is better since the right-hand side of the above range is $2/n$, as stated in the Kenig-Ruiz-Sogge estimates.

In \cite{Cue17}, Cuenin localized the eigenvalues of the perturbed Dirac operator in terms of the $L^p$-norm of the potential $V$, but in an unbounded region of the complex plane. Indeed, Theorem\til6.1.b of \cite{Cue17} states that, if $n\ge2$ and $|V|\in L^p$, with $p\ge n$, then any eigenvalue $z \in \rho(\D_0)$ of $\D_V$ satisfies
\begin{equation*}
	\left\lvert {\Im z}/{\Re z} \right\rvert^{(n-1)/p}
	|\Im z|^{1-n/p} 
	\le C 
	\n{V}_{L^p(\R^n)},
\end{equation*}
where $C$ is a constant independent on $z$ and $V$.
A similar result was proved by Fanelli and Krej\v ci\v r\'ik in \cite{FK19}, where they show that, for dimension $n=3$, $|V|\in L^3(\R^3)$ and $z \in \rho(\D_0) \cap \sigma_p(\R^n)$, then
\begin{equation*}
	\left( 1 + \frac{(\Re z)^2}{(\Re\sqrt{m^2-z^2})^2} \right)^{-1/2}
	<
	\left( {\pi}/{2} \right)^{1/3}
	\sqrt{1+e^{-1}+2e^{-2}} \n{V}_{L^3(\R^3)}.
\end{equation*} 
The advantage of the last result lies in the explicit condition which is easy to check in applications. However, still the eigenvalues are localized in an unbounded region around the continuous spectrum $\sigma(\D_0)=(-\infty,-m] \cup [m,+\infty)$ of the free Dirac operator $\D_0$.

Here, our main results try to generalize the one by Cuenin, Laptev and Tretter \cite{CLT14} in higher dimensions, recovering the enclosure of the eigenvalues of the massive ($m>0$) Dirac operator $\D_V$ in a compact region of the complex plane, imposing the smallness of the potential $V$ with respect to suitable mixed Lebesgue norms. In the case of the massless ($m=0$) Dirac operator, we obtain that the point spectrum of the perturbed operator $\D_{V}$ is empty, and then $\sigma(\D_{V})=\sigma(\D_0)=\R$. We also mention the recent paper \cite{CFK}, in which similar results are obtained by multiplication techniques.

Before to formalize our results in Theorems\til\ref{thm1}\til\&\til\ref{thm2}, we need to introduce the following notations used throughout the paper.

\begin{notations}\normalfont

We use the symbols $\sigma(H)$, $\sigma_p(H)$, $\sigma_e(H)$ and $\rho(H)$ respectively for the spectrum, the point spectrum, the essential spectrum and the resolvent of an operator $H$. Since for a non-selfadjoint closed operator there are various definitions for essential spectrum, we define
\begin{align*}
\sigma_e(H) =
\{z \in \C \colon \text{ $H-z$ is not a Fredholm operator} \},
\end{align*}
whereas the discrete spectrum is defined as
\begin{equation*}
\sigma_d(H) =
\{z \in \C \colon \text{$z$ is an isolated eigenvalue of $H$ of finite multiplicity}\}.
\end{equation*}
For $z\in \rho(H)$, we denote with $R_H(z) := (H-z)^{-1}$ the resolvent operator of $H$.
We recall also that 
\begin{gather*}
	\sigma(-\Delta)
	=\sigma_e(-\Delta)
	=[0,+\infty),
	\\
	\sigma(\D_0)
	=\sigma_e(\D_0)
	=(-\infty,-m]\cup[m,+\infty) . %\quad\text{for $m\ge0$.}
\end{gather*}

We use the symbol $(\cdot,\cdot)_{\H}$ for the inner product on the Hilbert space $\H=L^2(\R^n;\C^N)$, that is
\begin{equation*}
(\phi,\psi)_\H = \int_{\R^n} \phi^{\dagger} \cdot \psi \, dx
\end{equation*}
where $\cdot$ is the scalar product.

Fixed $j\in\{1,\dots,n\}$ and $x=(x_1,\dots,x_n)\in\R^n$, we denote
\begin{align*}
{\widehat{x}_j} &:= (x_1,\dots,x_{j-1},x_{j+1},\dots,x_n) \in \R^{n-1},
\\
(\o{x},{\widehat{x}_j}) &:=  (x_1,\dots,x_{j-1},\o{x},x_{j+1},\dots,x_n) \in\R^n.
\end{align*}
Define the mixed Lebesgue spaces $L^p_{x_j}L^q_{{\widehat{x}_j}}(\R^n)$ as the spaces of the functions with finite mixed-norm
\begin{equation*}
\n{f}_{L^p_{x_j}L^q_{{\widehat{x}_j}}} := \left( \int_{\R} \left(\int_{\R^{n-1}} |f(x_j,{\widehat{x}_j})|^q d{\widehat{x}_j} \right)^{p/q} dx_j \right)^{1/p},
\end{equation*}
where the obvious modifications occur for $p=\infty$ or $q=\infty$.

For any matrix--valued function $M \colon \R^n \to \C^{N\times N}$, 
%with an abuse of notation we indicate with the same symbol $M \colon \H \to \H$ the multiplication operator generated by $M$. Moreover, 
we set 
\begin{equation*}
\n{M}_{L^p_{x_j} L^q_{{\widehat{x}_j}}} := \n{|M|}_{L^p_{x_j} L^q_{{\widehat{x}_j}}}
\end{equation*}
where the function $|M|\colon \R^n\to\R$ is obtained considering the operator norm $|M(x)|$ of $M(x)$ for almost every fixed $x\in\R^n$. 

We also indicate with
\begin{align*}
[f *_{x_j} g](x) &:= \int_\R f(y_j,{\widehat{x}_j}) g(x_j-y_j,{\widehat{x}_j})dy_j,
\\
[\F_{x_j} f] (\xi_j, {\widehat{x}_j}) &:= \frac{1}{\sqrt{2\pi}} \int_{\R} e^{-i x_j \xi_j} f(x_j,{\widehat{x}_j}) dx_j,
\\
[\F^{-1}_{\xi_j} f] (x_j, {\widehat{x}_j}) &:= \frac{1}{\sqrt{2\pi}} \int_{\R} e^{i x_j \xi_j} f(\xi_j,{\widehat{x}_j}) d\xi_j,
\end{align*}
respectively the partial convolution respect to $x_j$, the partial Fourier transform respect to $x_j$ and its inverse. In a similar way one can define the partial (inverse) Fourier transform respect to ${\widehat{x}_j}$ and the complete (inverse) Fourier transform respect to $x$.

Finally, let us consider the function spaces
\begin{equation*}
	X \equiv X(\R^n) := \bigcap_{j=1}^{n} L^1_{x_j} L^2_{{\widehat{x}_j}}(\R^n),
	\qquad
	Y \equiv Y(\R^n) := \bigcap_{j=1}^{n} L^1_{x_j} L^\infty_{{\widehat{x}_j}}(\R^n),
\end{equation*}
with norms defined as
\begin{equation*}\label{Xnorm}
	\n{f}_X =
	\max_{j\in\{1,\dots,n\}}
	\n{ f }_{L^1_{x_j} L^2_{{\widehat{x}_j}}},
	\qquad
	\n{f}_Y =
	\max_{j\in\{1,\dots,n\}}
	\n{ f }_{L^1_{x_j} L^\infty_{{\widehat{x}_j}}}.
\end{equation*}
The dual space of $X$ is given (see e.g. \cite{BL12}) by
\begin{equation*}
	X^* \equiv X^*(\R^n) := \sum_{j=1}^{n} L^\infty_{x_j} L^2_{{\widehat{x}_j}}(\R^n),
\end{equation*}
with the norm
\begin{equation}\label{normX*}
	\n{f}_{X^*}
	:=
	\inf
	\left\{ \sum_{j=1}^{n} \n{f_j}_{L^\infty_{x_j} L^2_{{\widehat{x}_j}}} 
	\colon
	f=\sum_{j=1}^n f_j
	\right\}.
\end{equation}

\end{notations}

We can announce now our results.

\begin{theorem}
	\label{thm1}
	Let $m>0$.
	There exists a constant $C_0>0$, independent on $V$, such that if
	\begin{equation*}
	\n{V}_Y	< C_0,
	\end{equation*}
	then every eigenvalues $z \in \sigma_p(\D_{V})$ of $\D_{V}$ lies in the union 
	\begin{equation*}
	z \in \overline{B}_{R_0} (x^-_0) \cup \overline{B}_{R_0} (x^+_0)
	\end{equation*}
	of the two closed disks in $\C$ with centers in $x_0^-,x_0^+$ and radius $R_0$, with
	\begin{equation*}
	x^\pm_0 := \pm m\,\frac{\V^2+1}{\V^2-1},
	\quad
	R_0 := m\,\frac{2\V}{\V^2-1},
	\quad
	\V \equiv \V(V) := \left[\frac{(n+1)C_0}{\n{V}_Y} - n \right]^2 > 1.
	\end{equation*}
\end{theorem}

\begin{theorem}\label{thm2}
	Let $m=0$.
	There exists a constant $C_0>0$, independent on $V$, such that if
	\begin{equation*}
	\n{V}_Y	< C_0,
	\end{equation*}
	then $\D_V$ has no eigenvalues.
	%, i.e. $\sigma_p(\D_V)=\varnothing$. 
	%
	In particular, we have 
	$\sigma(\D_{V})=\sigma_{e}(\D_{V})=\R$.
\end{theorem}

%%%%%%

\begin{remark}
The crucial tool in our proof is a uniform resolvent estimate for the resolvent of the  free Dirac operator. This approach is inspired by \cite{Fra11}, where the result by Kenig, Ruiz and Sogge \cite{KRS87} was used for the same purpose. In our case, we prove in Section\til\ref{sec2} the following estimates, of independent interest:
\begin{equation*}
  \n{R_{-\Delta}(z)}_{X \to X^*} \le C |z|^{-1/2},
  \qquad
	\n{\partial_k R_{-\Delta}(z)}_{X \to X^*}  \le C.
\end{equation*}
and
\begin{equation*}%\label{RDest}
\n{ R_{\D_{0}}(z) }_{X \to X^*}
\le C \left[n + \left|\frac{z+m}{z-m}\right|^{\sgn(\Re z)/2}\right].
\end{equation*}
These can be regarded as precised resolvent estimates of Agmon--H\"{o}rmander type. Note also that similar uniform estimates, in less sharp norms, were proved earlier by the first and second Authors in \cite{DAnconaFanelli07-a,DAnconaFanelli08-a,EGG19}. 

In Section\til\ref{sec3}, we combine our uniform estimates with the Birman-Schwinger principle, enabling in Section\til\ref{sec4} to complete the proof of Theorems\til\ref{thm1}\til\&\til\ref{thm2}.
\end{remark}

\begin{remark}
	The following embedding for the space $Y$ under consideration hold:
	\begin{equation}\label{YinLn}
	Y \hookrightarrow L^{n,1}(\R^n) \hookrightarrow L^n(\R^n),
	\end{equation}
	where $L^{p,q}(\R^n)$ is the Lorentz space. Moreover, we have
	\begin{equation*}
	W^{1,1}(\R^n) 
	\hookrightarrow
	\bigcap_{j=1}^{n}  L^1_{{\widehat{x}_j}}
	L^\infty_{x_j}(\R^n)
	\hookrightarrow L^{n/(n-1),1}(\R^n),
	\end{equation*}
	where $W^{m,p}(\R^n)$ is the Sobolev space, and so, for the $2$-dimensional case, we get in particular
	\begin{equation*}
	W^{1,1}(\R^2) 
	\hookrightarrow
	Y = 
	L^1_{x_1} L^\infty_{x_2}(\R^2)
	\cap
	L^1_{x_2} L^\infty_{x_1}(\R^2)
	\hookrightarrow 
	L^{2,1}(\R^2)
	\hookrightarrow 
	L^{2}(\R^2).
	\end{equation*}
	We refer to the papers by Fournier \cite{Fou87}, Blei and Fournier \cite{BF89} and Milman \cite{Mil91} for these inclusions.
	%
	%As an easy example, potential such that $|V| \sim (1+|x|)^{-\alpha}$ with $\alpha>1$ are contemplated by our theorem.
\end{remark}

%%%%%%%%%%%%%%%%%
\section{Uniform resolvent estimates}\label{sec2}

Let us start %defining a particular decomposition for a fixed $f\in X$, i.e. functions $\{f_j\}_{j\in\{1,\dots,n\}}$ such that $\sum_{j=1}^n f_j =f$. 
fixing constants $r,R,\delta>0$ such that
\begin{equation*}\label{cond_const}
1<r<R,
\qquad
\sqrt{R^2-1} < \delta <1,
\end{equation*}
and consider
the open cover $\S=\{\S^+_{ j}, \S^-_{ j}, \S_\infty \}_{j\in\{1,\dots,n\}}$ of the space $\R^n$ defined by
\begin{equation*}
\S^\pm_{ j} = \{ \xi\in\R^n \colon \pm \xi_j > \delta|\widehat{\xi}_j|, \,\, |\xi|<R \}, 
\qquad
\S_\infty = \{ \xi\in\R^n \colon |\xi|>r \}.
\end{equation*}
We can find a smooth partition of unity $\{ \chi^+_{ j}, \chi^-_{ j} , \chi_\infty \}_{j\in\{1,\dots,n\}}$ subordinate to $\S$, i.e. a family of smooth positive functions such that
\begin{equation*}
\supp \chi^\pm_{ j} \subset \S^\pm_{ j},
\quad
\supp \chi_{\infty} \subset \S_{\infty}, 
\quad
\chi_{\infty} + \sum_{j=1}^n [ \chi^+_{ j} + \chi^-_{ j} ] \equiv 1.
\end{equation*}
From these, define the smooth partition of unity $\chi=\{\chi_j\}_{j\in\{1,\dots,n\}}$, with
\begin{equation}\label{def_chij}
\chi_j := \chi_j^+ +\chi_j^- + \frac{1}{n}\chi_{ \infty},
\end{equation}
and hence, for $j\in\{1,\dots,n\}$, the Fourier multipliers
\begin{equation*}\label{fj}
\chi_j(|z|^{-1/2} D) f =
\F^{-1}_\xi[\chi_j(|z|^{-1/2}\xi) \, \F_x f].
%\left\{
%\begin{aligned}
%&\F^{-1}_\xi[\chi_j(|z|^{-1/2} \xi) \, \F_x f], 
%&\quad\text{if $z\neq0$,} \\
%&\frac{1}{n}f, 
%&\quad\text{if $z=0$.}
%\end{aligned}
%\right.
\end{equation*}
Note in particular that
\begin{equation}\label{sumchij}
	\sum_{j=1}^n \chi_j(|z|^{-1/2} D) f = f.
\end{equation}
Therefore, the following estimates hold.

\begin{lemma}\label{lem1}
	For every $z\in \rho(-\Delta) = \C\setminus[0,+\infty)$,
	$f \in L_{x_j}^1 L_{{\widehat{x}_j}}^2$ and $j,k \in\{1,\dots,n\}$, we have that
	\begin{align*}%\label{R0}
	\n{ \chi_j \left( |z|^{-1/2} D \right) R_{-\Delta}(z) f }_{L_{x_j}^\infty L_{{\widehat{x}_j}}^2}
	&\le C \, |z|^{-1/2} 
	\n{f}_{L_{x_j}^1 L_{{\widehat{x}_j}}^2},
	\\
	%\label{R0grad}
	\n{ \chi_j \left( |z|^{-1/2} D \right) \partial_k  R_{-\Delta}(z) f }_{L_{x_j}^\infty L_{{\widehat{x}_j}}^2}
	&\le C
	\n{f}_{L_{x_j}^1 L_{{\widehat{x}_j}}^2},
	\end{align*}		
	where $\{\chi_j\}_{j\in \{1,\dots,n\}}$ are defined in \eqref{def_chij} and $C>0$ does not depend on $z$.
	In particular, it follows that
	\begin{align*}
	\n{R_{-\Delta}(z)}_{X \to X^*} & \le C |z|^{-1/2},
	\\
	\n{\partial_k R_{-\Delta}(z)}_{X \to X^*} & \le C.
	\end{align*}
\end{lemma}

%%%%%%

\begin{lemma}\label{lem2}
	For every $z\in\rho(\D_0) = \C\setminus\{\zeta\in\R \colon |\zeta| \ge m\}$, $f \in L_{x_j}^1 L_{{\widehat{x}_j}}^2$ and $j\in \{1,\dots,n\}$ we have that
	\begin{equation*}\label{RDest}
	\n{ \chi_j \left( |z^2-m^2|^{-1/2} D \right) R_{\D_{0}}(z) f }_{L_{x_j}^\infty L_{{\widehat{x}_j}}^2}
	\le C \left[n + \left|\frac{z+m}{z-m}\right|^{\sgn(\Re z)/2}\right]
	\n{f}_{L_{x_j}^1 L_{{\widehat{x}_j}}^2},
	\end{equation*}	
	where $\{\chi_j\}_{j\in \{1,\dots,n\}}$ are defined in \eqref{def_chij} 
	and $C>0$ is the same as in Lemma\til\ref{lem1}. In particular, it follows that
	\begin{equation*}%\label{RDest}
	\n{ R_{\D_{0}}(z) }_{X \to X^*}
	\le C \left[n + \left|\frac{z+m}{z-m}\right|^{\sgn(\Re z)/2}\right].
	\end{equation*}
\end{lemma}

\begin{remark}
	Before to proceed further, we give an heuristic explanation for the choice of the localization in the frequency domain via the Fourier multiplier $\chi_j(|z|^{-1/2}D)$. Since the symbol $(|\xi|^2-z)^{-1}$ of the resolvent $R_{-\Delta}(z)$ blows-up as $z \to \zeta\ge0$, our trick is to use the norms $L^\infty_{x_j} L^2_{{\widehat{x}_j}}$ for $j\in\{1,\dots,n\}$, which allows us to restrict the problem from the spherical surface $\{\xi\in\R^n \colon |\xi|=|z|^{-1/2}\}$ to the \lq\lq equator'' given by $\{ \xi\in\R^n \colon \xi_j=0, |\widehat{\xi}_j|=|z|^{-1/2} \}$. We then avoid these regions thanks to the smooth functions $\chi_j$. 
\end{remark}

\begin{proof}[Proof of Lemma\til\ref{lem1}.]
The \lq\lq in particular'' part trivially follows from \eqref{sumchij}
and from the definitions of the norms on $X$ and $X^*$.
	
For the simplicity, from now on $C>0$ will stand for a generic positive constant independent on $z$ and which can change from line to line. Clearly, by scaling, we can consider only unitary $z\in\C\setminus \{1\}$.
Thus it is sufficient to prove that
\begin{equation*}
\n{ \chi_j(D) \partial_k^s  R_{-\Delta}(z) f }_{L_{x_j}^\infty L_{{\widehat{x}_j}}^2}
\le C
\n{f}_{L_{x_j}^1 L_{{\widehat{x}_j}}^2},
\end{equation*}	
where $|z|=1$, $s\in \{0,1\}$, $\partial^0_k=1, \partial^1_k=\partial_k$ and $j,k\in\{1,\dots,n\}$.
This is equivalent to
	\begin{equation}\label{estz=1}
	\n{ \F^{-1}_{\xi} \left( \frac{\xi_k^s  \, \chi_j(\xi) }{|\xi|^2-\l-i\e} \F_{x}f \right) }_{L_{x_j}^\infty L_{{\widehat{x}_j}}^2}
	\le C
	\n{f}_{L_{x_j}^1 L_{{\widehat{x}_j}}^2},
	\end{equation}
where we have settled $z=\l+i\e$, $\l^2+\e^2=1$ and $z\neq1$. We proceed splitting $\chi_j$ in the functions which appear in its definition \eqref{def_chij}, localizing ourselves in the regions of the frequency domain near the unit sphere, i.e. $\S^\pm_j$, and far from it, i.e. $\S_\infty$.

\textit{Estimate on $\S^\pm_j$.}
We want to prove
\begin{equation}\label{eq_Sj}
\n{ \F^{-1}_{\xi} \left( \frac{\xi_k^s \chi_j^\pm(\xi)}{|\xi|^2-\l-i\e} \F_{x}f \right) }_{L_{x_j}^\infty L_{{\widehat{x}_j}}^2}
\le C
\n{f}_{L^2_{{\widehat{x}_j}} L^1_{x_j}}.
\end{equation}
Let us define the family of operators
\begin{equation*}
T^\pm_j \colon L^p_{x_j}L^2_{{\widehat{x}_j}} \to  L^p_{x_j}L^2_{{\widehat{x}_j}}, \quad
f \mapsto T^\pm_j f := \F_\xi^{-1} \left(\hat{f} \circ \Phi \right),
\end{equation*}
where
\begin{equation*}
\Phi(\xi) := (\xi_j+\varphi(\widehat{\xi}_j), \widehat{\xi}_j),
\quad
\vp(\widehat{\xi}_j) := \pm\sqrt{|1-|\widehat{\xi}_j|^2|}.
\end{equation*}
Roughly speaking, the operators $T^\pm_j$ \lq\lq flatten'' in the frequency domain the hemisphere of the unitary sphere, namely $\{ \xi\in\R^n \colon |\xi|=1, \pm\xi_j > 0 \}$. 
Writing more explicitly the introduced operators, we have that
\begin{equation*}
\begin{split}
T^\pm_j f(x)
&= \frac{1}{(2\pi)^{n/2}} \int_{\R^n} e^{i x\cdot\xi} \hat{f}(\xi_j+\varphi(\widehat{\xi}_j), \widehat{\xi}_j) d\xi \\
&= \frac{1}{(2\pi)^{n}} \int_{\R^n} e^{i x\cdot\xi} \int_{\R^n} f(y) e^{-iy \cdot(\xi_j+\varphi(\widehat{\xi}_j), \widehat{\xi}_j)} dy d\xi \\
&= \frac{1}{(2\pi)^{n}} 
\int_{\R^{n-1}} e^{i {\widehat{x}_j}\cdot\widehat{\xi}_j} 
\int_{\R^{n-1}} e^{-i y'\cdot\widehat{\xi}_j}
\int_{\R}\int_{\R} f(y) 
e^{i (x_j-y_j)\xi_j -i y_j\varphi(\widehat{\xi}_j)}  
dy_j d\xi_j dy'd\widehat{\xi}_j \\
&= \frac{1}{2\pi} 
\F^{-1}_{\widehat{\xi}_j} \F_{y'} \left(
e^{-i x_j \varphi(\widehat{\xi}_j)}
\int_{\R}\int_{\R} f(y) 
e^{i (x_j-y_j)\xi_j}  
dy_j d\xi_j \right)\\
&= \F^{-1}_{\widehat{\xi}_j} \F_{y'} \left(
e^{-i x_j \varphi(\widehat{\xi}_j)}  
f(x_j,y')
\right)
\end{split}
\end{equation*}
where we exploited the substitution $\xi_j\mapsto\xi_j-\varphi(\widehat{\xi}_j)$ in the fourth equality. Hence, applying the Plancherel Theorem two times, we obtain that $T^\pm_j$ are isometries respect to the norm $L^p_{x_j}L^2_{{\widehat{x}_j}}$, i.e. for $p \in [1,+\infty]$ it holds the relation
\begin{equation}\label{iso}
\n{T^\pm_j f}_{L^p_{x_j}L^2_{{\widehat{x}_j}}} = \n{f}_{L^p_{x_j}L^2_{{\widehat{x}_j}}}	.
\end{equation}

Thanks to this equality, we can write
\begin{equation*}\label{est_conv}
\begin{split}
\n{ \F^{-1}_{\xi} \left( \frac{\xi_k^s \chi_j^\pm(\xi)}{|\xi|^2-\l-i\e} \F_{x}f \right) }_{L_{x_j}^\infty L_{{\widehat{x}_j}}^2}
&= \n{ T^\pm_j \F^{-1}_{\xi} \left( \frac{\xi_k^s \chi_j^\pm(\xi)}{|\xi|^2-\l-i\e} \F_{x}f \right) }_{L_{x_j}^\infty L_{{\widehat{x}_j}}^2} 
\\
%&= \n{ T^\pm_j\F^{-1} \left( \frac{\wh{f^\pm_{ j}}}{|\xi|^2-1-i\e}\right)}_{L^p_{x_j}L^2_{{\widehat{x}_j}}} \\
&= \n{ \F_\xi^{-1} \left( \frac{ (\xi_k^s \chi_j^\pm ) \circ \Phi}{|\Phi|^2-\l-i\e}
	\, \wh{T^\pm_j f} \right)}_{L^\infty_{x_j}L^2_{{\widehat{x}_j}}} 
\\
&= \frac{1}{\sqrt{2\pi}}
\n{ {a_{\l,\e}}(D) \psi *_{x_j}
	\F_{\xi_j}^{-1} \left( \frac{\wh{T^\pm_j f} }{\xi_j-i|\e|} \right)}_{L^\infty_{x_j}L^2_{\widehat{\xi}_j}}
\\
&\le \frac{1}{\sqrt{2\pi}}
\n{ {a_{\l,\e}}(D) \psi }_{L^1_{x_j}L^{\infty}_{\widehat{\xi}_j}}
\n{ \F_\xi^{-1} \left( \frac{\wh{T^\pm_j f} }{\xi_j-i|\e|} \right)}_{L^{\infty}_{x_j}L^2_{\widehat{\xi}_j}}
\end{split}
\end{equation*}
where the last inequality follows from the Young's Theorem and
\begin{gather*}
{a_{\l,\e}}(D) \psi
=\F^{-1}_{\xi_j} \left( a_{\l,\e} \F_{x_j}(\psi) \right) ,
\\
a_{\l,\e} (\xi) := 
\frac{(\xi_j-i|\e|) \, 
\left(\xi_k \pm \delta_{k,j}\sqrt{1-|\widehat{\xi}_j|^2}\right)^s}
{\xi_j \left(\xi_j \pm 2\sqrt{1-|\widehat{\xi}_j|^2} \right)+1-\l-i\e} 
\,
\sqrt{(\chi^\pm_j \circ \Phi)(\xi)} ,
\\
\psi(x_j,\widehat{\xi}_j) = \F_{\xi_j}^{-1} \left(\sqrt{(\chi^\pm_j \circ \Phi)(\xi)}\right)
.
\end{gather*}
Observe that we dropped the absolute value in the definition of $\varphi$, namely it results $\sqrt{|1-|\widehat{\xi}_j|^2|}=\sqrt{1-|\widehat{\xi}_j|^2}$, because $\supp\{\chi_j^\pm \circ \Phi\} \subset \{\xi\in\R^n \colon |\widehat{\xi}_j|\le1\}$, thanks to the definition of $\S^\pm_j$ and from the assumption $\delta\ge\sqrt{R^2-1}$. Despite of the involute definition, it is simple to see that ${a_{\l,\e}}(D) \psi \in \mathfrak{S}$, where $\mathfrak{S}$ is the space of Schwartz functions, since ${a_{\l,\e}}(D) \psi$ is the inverse Fourier transform of a smooth compactly supported function. 
Moreover, considering ${a_{\l,\e}}(D) \psi$ as a pseudo-differential operator of symbol $a_{\l,\e}$ applied to the Schwartz function $\psi$, let us to observe that, since letting $\l+i\e\to1$ we have
\begin{equation*}
	\lim_{\l+i\e\to1} a_{\l,\e} (\xi) =
	\frac{\left(\xi_k \pm \delta_{k,j}\sqrt{1-|\widehat{\xi}_j|^2}\right)^s}{\xi_j \pm 2\sqrt{1-|\widehat{\xi}_j|^2}} 
	\,
	\sqrt{\chi_j^\pm \left(\xi_j\pm\sqrt{1-|\widehat{\xi}_j|^2},\widehat{\xi}_j \right)} =: a(\xi) \in\mathfrak{S}
\end{equation*}
pointwisely, then ${a_{\l,\e}}(D)\psi \to {a}(D)\psi$ in $\mathfrak{S}$, and so 
\begin{equation*}
\lim_{\l+i\e\to1}
\n{{a_{\l,\e}}(D) \psi}_{L^1_{x_j}L^\infty_{\widehat{\xi}_j}}
= \n{a(D)\psi}_{L^1_{x_j}L^\infty_{\widehat{\xi}_j}} < +\infty.
\end{equation*}
Thus, $\n{{a_{\l,\e}}(D) \psi}_{L^1_{x_j}L^\infty_{\widehat{\xi}_j}}$ is uniformly bounded respect to $z\in\C$ unitary.

Hence we have obtained
\begin{equation}\label{est1}
\n{ \F^{-1}_{\xi} \left( \frac{\xi_k^s \chi_j^\pm(\xi)}{|\xi|^2-\l-i\e} \F_{x}f \right) }_{L_{x_j}^\infty L_{{\widehat{x}_j}}^2}
\le C
\n{ \F_\xi^{-1} \left( \frac{\wh{T^\pm_j f} }{\xi_j-i|\e|} \right)}_{L^{\infty}_{x_j}L^2_{\widehat{\xi}_j}}.
\end{equation}
By Plancherel and Young Theorems, and by equality \eqref{iso}, we have
\begin{equation*}
\begin{split}
\sqrt{2\pi} \n{ \F_{\xi_j}^{-1} \left( \frac{\wh{T^\pm_j f} }{\xi_j-i|\e|} \right)}_{L^\infty_{x_j}L^2_{\widehat{\xi}_j}} 
&= \n{ \F_{\xi_j}^{-1} \left( \frac{1 }{\xi_j-i|\e|}\right) *_{x_j} \F_{{\widehat{x}_j}} (T^\pm_j f) }_{L^\infty_{x_j}L^2_{\widehat{\xi}_j}} 
\\
& = \n{ i e^{-|\e| x_j } \theta *_{x_j} 
	\F_{{\widehat{x}_j}} (T^\pm_j f) }_{L^\infty_{x_j}L^2_{\widehat{\xi}_j}}
\\
& \le \n{ e^{-|\e| x_j } \theta *_{x_j} 
	\n{T^\pm_j f}_{L^2_{{\widehat{x}_j}}} }_{L^\infty_{x_j}} 
\\
&\le \n{ e^{-|\e| x_j } \theta }_{L^\infty_{x_j}} \n{f}_{L^1_{x_j}L^2_{{\widehat{x}_j}}}
\\
&= \n{f}_{L^1_{x_j}L^2_{{\widehat{x}_j}}}
,
\end{split}
\end{equation*}
where $\theta \equiv \theta(x_j)$ is the Heaviside function. Thus, inserting in \eqref{est1}, we get \eqref{eq_Sj}.

%%%%%%%%%%%%%

\textit{Estimate on $\S_\infty$.}
We want to prove now
\begin{equation}\label{eq_Sinfty}
\n{ \F^{-1}_{\xi} \left( \frac{\xi_k^s \chi_\infty(\xi) }{|\xi|^2-\l-i\e} \F_{x}f \right) }_{L_{x_j}^\infty L_{{\widehat{x}_j}}^2}
\le C
\n{f}_{L^2_{{\widehat{x}_j}} L^1_{x_j}}.
\end{equation}
We need to distinguish three cases depending on whether we are localized in the regions defined by
\begin{align*}
	\Cyl^{1}_{R,j} &:= \{ \xi\in\R^n \colon |\widehat{\xi}_j|>R \},
	\\
	\Cyl^{2}_{R,j} &:= \{ \xi\in\R^n \colon |\widehat{\xi}_j|\le R, |\xi_j|\le 2R \},
	\\
	\Cyl^{3}_{R,j} &:= \{ \xi\in\R^n \colon |\widehat{\xi}_j|\le R, |\xi_j|> 2R \}.
\end{align*}
%Denoting with $\chrt_{X}$ the characteristic function of a set ${X}$, i.e. $\chrt_{X}(\xi)=1$ if $\xi\in{X}$ and $\chrt_{X}(\xi)=0$ otherwise, we set then
%\begin{equation*}
%	\chi^1_\infty = \chrt_{C^1_{R,j}},
%	\quad
%	\chi^2_\infty = \chrt_{C^2_{R,j}}\chi_\infty,
%	\quad
%	\chi^3_\infty = \chrt_{C^3_{R,j}}.
%\end{equation*}
Let us set then
\begin{align*}
	\chi^1_\infty(\xi) &:=
	\begin{cases}
	\makebox[\longtext][l]{1} &\text{if $|\widehat{\xi}_j|>R$,}\\
	\makebox[\longtext][l]{0} &\text{otherwise,}
	\end{cases}
	\\
	\chi^2_\infty(\xi) &:=
	\begin{cases}
	\chi_\infty(\xi) &\text{if $|\widehat{\xi}_j|\le R$ and $|\xi_j|\le 2R$,}\\
	\makebox[\longtext][l]{0} &\text{otherwise,}
	\end{cases}
	\\
	\chi^3_\infty(\xi) &:=
	\begin{cases}
	\makebox[\longtext][l]{1} &\text{if $|\widehat{\xi}_j|\le R$ and $|\xi_j|> 2R$,}\\
	\makebox[\longtext][l]{0} &\text{otherwise,}
	\end{cases}
\end{align*}
and observe that $\chi_\infty = \chi^1_\infty + \chi^2_\infty + \chi^3_\infty$, since $\chi_\infty \equiv 1$ for $|\xi|>R$, from the requirements on the cover $\S$ and the partition $\chi$. 

By Plancherel Theorem and H\"older's, Young's and Minkowski's integral inequalities, for $h\in\{1,2,3\}$ we infer
\begin{equation*}
\n{ \F^{-1}_{\xi} \left( \frac{\xi_k^s \chi^h_\infty(\xi)}{|\xi|^2-\l-i\e} \F_{x}f \right) }_{L_{x_j}^\infty L_{{\widehat{x}_j}}^2}
\le
{C_h}
\n{f}_{L^1_{x_j} L^2_{{\widehat{x}_j}}}
\end{equation*} 
with
\begin{equation}\label{Ci}
	C_h :=
	\frac{1}{\sqrt{2\pi}}
	\n{ \F^{-1}_{\xi_j} \left( \frac{\xi_k^s \, \chi^h_\infty(\xi) }{\xi_j^2 + \s^2} \right) }_{L^\infty_{x_j} L^\infty_{\widehat{\xi}_j}},
	\quad
	\s:=\sqrt{|\widehat{\xi}_j|^2-\l-i\e}.
\end{equation}
Here and below, we always consider the principal branch of the complex square root function.
Clearly, if we prove that $C_h$ for $h\in\{1,2,3\}$ is bounded uniformly respect to $\l,\e$, we recover \eqref{eq_Sinfty}.

\textit{Estimate on $\Cyl^1_{R,j}$.} Observing that $\chi_{ \infty}^1(\xi) \equiv 
\chi_{ \infty}^1(\widehat{\xi}_j)$ and
noting that
\begin{align*}
\Re\sigma &= \sqrt{ \frac{|\sigma|^2+|\widehat{\xi}_j|^2-\l}{2} } >0,
\end{align*}
we can compute explicitly the Fourier transforms:
\begin{itemize}
	\item if $k \neq j$, then
	\begin{equation*}
	\begin{split}
		C_1 =  
		\n{ \chi_\infty^1(\widehat{\xi}_j)
			\xi_k^s \frac{e^{-\s|x_j|}}{2\s} }_{L^\infty_{x_j} L^\infty_{\widehat{\xi}_j}}
		&\le
		\n{ \chi_\infty^1(\widehat{\xi}_j) |\widehat{\xi}_j|^s \frac{e^{- \Re\sigma |x_j|}}{2|\s|} }_{L^\infty_{x_j} L^\infty_{\widehat{\xi}_j}}
		\\&\le
		\sup_{|\widehat{\xi}_j|>R} \frac{|\widehat{\xi}_j|^s}{2(|\widehat{\xi}_j|^4-2\l|\widehat{\xi}_j|^2+1)^{1/4}}
		\\&\le
		\begin{cases}
		\frac{R^s}{2\sqrt{R^2-1}} & \text{if $\l>0$,}
		\\
		1/2 & \text{if $\l\le0$;}
		\end{cases}
	\end{split}
	\end{equation*}
	\item if $s=1$, $k = j$, then
	\begin{equation*}
		C_1 =  
		\n{ \chi_\infty^1(\widehat{\xi}_j) \frac{i}{2} \sgn\{x_j\} e^{-\s|x_j|} }_{L^\infty_{x_j} L^\infty_{\widehat{\xi}_j}}
		\le \frac{1}{2}.
	\end{equation*}
\end{itemize}

\textit{Estimate on $\Cyl^2_{R,j}$.} Expliciting the definition of inverse Fourier transform in \eqref{Ci} and from the fact that $\chi^2_\infty(\xi) = 0$ when $|\xi|<r$, we can trivially see that
\begin{equation*}
C_2 \le
\frac{1}{2\pi}
\n{  \int_{-\infty}^{+\infty} 
	\frac{ |e^{i x_j \xi_j}| \, |\xi_k^s| \, \chi^2_\infty(\xi)}{||\xi|^2-\l|}
	d\xi_j
}_{L^\infty_{x_j} L^\infty_{\widehat{\xi}_j}}
\le
\frac{(2R)^s}{2\pi}
\n{ 
	\frac{\chi^2_\infty(\xi)}{|\xi|^2-1}
}_{L^\infty_{\widehat{\xi}_j} L^1_{\xi_j}}
\end{equation*}
which is finite since $\chi^2_\infty$ is compactly supported due to its definition.

\textit{Estimate on $\Cyl^3_{R,j}$.} 
Expliciting the inverse Fourier transform in \eqref{Ci}, recalling the definition of $\chi^3_\infty$ and exploiting the substitution $\xi_j \mapsto \sgn\{x_j\}\xi_j$, we have
\begin{align*}
C_3 &=
\frac{1}{2\pi}
\n{ (1-\chi^1_\infty)(\widehat{\xi}_j) \int_{|\xi_j|>R} e^{i |x_j| \xi_j} \frac{\xi_k^s }{\xi_j^2 + \s^2} d\xi_j }_{L^\infty_{x_j} L^\infty_{\widehat{\xi}_j}}
\\
&=
\frac{1}{2\pi}
\n{ \int_{|\xi_j|>R} \psi(x_j,\widehat{\xi}_j,\xi_j) d\xi_j }_{L^\infty_{x_j} L^\infty_{\widehat{\xi}_j}}
\end{align*}
where, for fixed $\widehat{\xi}_j, x_j$,
%such that $|\widehat{\xi}_j|<R$ and $x_j\neq0$
the complex function $\psi(x_j,\widehat{\xi}_j,\cdot) \colon\C\to\C$ is defined by
\begin{equation*}
	\psi(x_j,\widehat{\xi}_j,w) :=
	\left\{
	\begin{aligned}
	& (1-\chi^1_\infty)(\widehat{\xi}_j) \frac{\xi_k^s}{w^2 + \sigma^2}e^{i |x_j| w}, &&\text{if $k\neq j$,}
	\\
	& (1-\chi^1_\infty)(\widehat{\xi}_j) \frac{w}{w^2 + \sigma^2} e^{i |x_j| w}, &&\text{if $s=1$, $k=j$,}
	\end{aligned}
	\right.
\end{equation*}
which is holomorphic in $\C\setminus\{w_-,w_+\}$, where $w_\pm = \pm i\sigma$. Observe that $\psi\equiv0$ for $|\widehat{\xi}_j|>R$, and if $|\widehat{\xi}_j|\le R$ we have
\begin{equation}\label{|sigma|}
|w_\pm| = |\sigma| = \sqrt[4]{ (|\widehat{\xi}_j|^2-\l)^2 +\e^2 } < \sqrt{2}R .
\end{equation}
Define, for a radius $A>0$, the semicircle 
$
	\gamma_A := \{Ae^{i\theta} \colon \theta\in [0,\pi] \}
$
in the upper half complex plane.
Fixing $\rho>R$, by the Residue Theorem, we get
\begin{equation*}
	\left( 
	\int_{[-\rho,-2R]} - \int_{\gamma_{2R}} + \int_{[2R,\rho]} + \int_{\gamma_\rho}
	\right) \psi(x_j,\widehat{\xi}_j,w) dw = 0.
\end{equation*}
Observing that we can consider $x_j\neq0$, letting $\rho\to +\infty$ we can apply the Jordan's lemma to the integral on the curve $\gamma_\rho$, finally getting
\begin{align*}
	C_3 &=
	\frac{1}{2\pi} \n{ \int_{\gamma_{2R}} \psi(x_j, \widehat{\xi}_j, w) dw }_{L^\infty_{x_j} L^\infty_{\widehat{\xi}_j}}
	\\
	&\le
	\frac{(2R)^s}{2\pi} \n{ (1-\chi^1_\infty)(\widehat{\xi}_j) \int_{0}^{\pi}  \frac{d\theta }{|4R^2e^{2i\theta} + \sigma^2|} }_{L^\infty_{\widehat{\xi}_j}}
	\\
	&\le
	(2R)^{s-2}
	%\frac{(2R)^s}{2\pi} \frac{\pi}{2R^2}
\end{align*} 
where we used the relation \eqref{|sigma|}.

Summing all up, we can finally recover the desired estimate \eqref{estz=1}, where the positive constant $C$ does not depend on $\l,\e$, but only on $R$ and the partition $\chi$.
\end{proof}

We can prove now Lemma\til\ref{lem2}, as a straightforward corollary of Lemma\til\ref{lem1}.

\begin{proof}[Proof of Lemma\til\ref{lem2}]
	Again, the \lq\lq in particular'' part follows from
	\eqref{sumchij}
	and the definitions of the $X$ and $X^*$ norms.
	
	From the anticommutation relations \eqref{clifford} we infer, for every $z\in\C$, 
	\begin{equation*}\label{trick}
	(\D_{0} - zI_N)(\D_{0} + zI_N) = (-\Delta + m^2 -z^2)I_N.
	\end{equation*}
	So, thanks to this well-known trick, for $z\in\rho(\D_0)$ we can write
	\begin{equation*}%\label{RDandR0}
	R_{\D_{0}}(z) = (\D_{0} +zI_N) R_{-\Delta}(z^2-m^2)I_N.
	\end{equation*}
	Set $f_j = \chi_j(|z^2-m^2|^{-1/2}D)f$ for the simplicity. Exploiting Lemma~\ref{lem1}, it is easy to get
	\begin{equation*}
	\begin{split}
	\n{ R_{\D_{0}}(z) f_j}_{L^{\infty}_{x_j} L^2_{\widehat{x}_j}} 
	\le&\,
	\n{ \sum_{k=1}^n \alpha_k \partial_k R_{-\Delta}(z^2-m^2) f_j}_{L^{\infty}_{x_j} L^2_{\widehat{x}_j}}
	\\
	&+\n{ (m\a_0 + z I_{N}) R_{-\Delta}(z^2-m^2) f_j}_{L^{\infty}_{x_j} L^2_{\widehat{x}_j}}
	\\
	\le&\, 
	\sum_{k=1}^n
	\n{ \partial_k R_{-\Delta}(z^2-m^2) f_j}_{L^{\infty}_{x_j} L^2_{\widehat{x}_j}}
	\\
	&+ \max\{|z+m|, |z-m|\} \n{ R_{-\Delta}(z^2-m^2) f_j}_{L^{\infty}_{x_j} L^2_{\widehat{x}_j}} 
	\\
	\le&\, 
	C \left[n + \left|\frac{z+m}{z-m}\right|^{\sgn(\Re z)/2}\right]
	\n{f}_{L^{1}_{x_j} L^2_{\widehat{x}_j}} .
	\end{split}
	\end{equation*}
	as claimed.
\end{proof}

%%%%%%%

\section{The Birman-Schwinger principle}\label{sec3}

In this section, working in great generality, we precisely define the closed extension of a perturbed operator with a factorizable potential, formally defined as $H_0+B^*A$. Next, we state an abstract version of the Birman-Schwinger principle.
We will follow the approach of Kato \cite{Kato66} and Konno and Kuroda \cite{KK66}.

	Let $\H, \H'$ be Hilbert spaces and consider the densely defined, closed, linear operators
	\begin{equation*}
	%\begin{aligned}
	H_0 \colon \dom(H_0) \subseteq \H \to \H,
	\quad
	A \colon \dom(A) \subseteq \H \to \H', 
	\quad
	B \colon \dom(B) \subseteq \H \to \H', 
	%\end{aligned}
	\end{equation*}
	such that $\rho(H_0) \neq \varnothing$
	and
	\begin{equation*}
		\dom(H_0) \subseteq\dom(A),
		\quad
		\dom(H_0^*) \subseteq\dom(B).
	\end{equation*}
	For the simplicity, we assume also that $\sigma(H_0) \subset \R$ and $\sigma_p(H_0)=\varnothing$.
	For $z\in\rho(H_0)$, denote by $R_{H_0}(z)=(H_0-z)^{-1}$ the resolvent operator of $H_0$.
	
As a warm-up, to sketch the idea of the principle, let us firstly consider the case of bounded operators $A$ and $B$. Thus, $H=H_0+B^*A$ is well-defined as a sum operator. Moreover, if $z\in\rho(H_0)$, we can define the bounded operator $Q(z)=A(H_0-z)^{-1}B^*$. It is easy to check that $z\in\sigma_p(H)\cap\rho(H_0)$ implies $-1\in\sigma_p(Q(z))$, and so \mbox{$\n{Q(z)}_{\H'\to\H'} \ge 1$}. Hence, a bound for the norm of the Birman-Schwinger operator $Q(z)$ give us information on the localization of the non-embedded eigenvalues of $H$.

Let us return to the general setting of an unbounded potential $B^*A$. Furthermore, we study also the case of embedded eigenvalues. To afford these, we assume a stronger set of hypotheses respect to the ones in \cite{KK66}, which instead will be proved here in Lemma\til\ref{lemBSop}.

\begin{hypothesis}[\textbf{Hyp}]\label{hyp}
	Let $X$ a complex Banach space of function on $\R^n$ such that the following estimates hold true:
	\begin{align*}
	\n{AR_{H_0}(z)}_{X\to \H'} &\le \alpha \Lambda(z),
	\\
	\n{B^*}_{\H' \to X} &\le \beta,
	\end{align*}
	where $\alpha,\beta>0$, $\Lambda(z)$ is finite for $z\in\rho(H_0)$ and continue up to $\sigma(H_0) \setminus \Omega$, for some set $\Omega \subseteq \sigma(H_0)$, and there exists $z_0\in\rho(H_0)$ such that $\Lambda(z_0)<(\alpha\beta)^{-1}$. 
\end{hypothesis}

\begin{lemma}[Birman-Schwinger operator]
	\label{lemBSop} 
	Assume (Hyp). Then, for $z\in\rho(H_0)$, the operator $AR_{H_0}(z)B^*$, densely defined on $\dom(B^*)$, has a closed extension $Q(z)$ in $\H'$,
		\begin{equation*}
		Q(z) = \overline{AR_{H_0}(z)B^*} %\in \mathcal{B}(\H'),
		\end{equation*}
	with norm bounded by
		\begin{equation}\label{Qnorm}
		\n{Q(z)}_{\H'\to\H'} \le \alpha\beta\Lambda(z).
		\end{equation}
	%where $\mathcal{B}(\H')$ is the Banach space of bounded linear operators in $\H'$ and $\overline{T}$ is the closure of a bounded and densely defined operator $T$.
	
	Moreover, there exists $z_0 \in\rho(H_0)$ such that $-1 \in \rho(Q(z_0)).$
\end{lemma}
\begin{proof}
	For $z\in\rho(H_0)$ and $\varphi\in\dom(B^*)$, we have
	\begin{equation*}
	\n{AR_{H_0}(z)B^* \varphi}_{\H'\to\H'} \le \n{AR_{H_0}(z)}_{X\to \H'} \n{B^*}_{\H' \to X} \n{\varphi}_{\H'}
	\end{equation*}
	and so by density \eqref{Qnorm}.
	In particular, for $z_0\in\rho(H_0)$ such that $\Lambda(z_0)<(\alpha\beta)^{-1}$, we get $\n{Q(z_0)}_{\H'\to\H'} < 1$. Hence, from Neumann series there exists $(1+Q(z_0))^{-1}$, and so $-1\in\rho(Q(z_0))$.
\end{proof}

%%%%

Let us collect some useful facts in the next lemma.

\begin{lemma}\label{lem-rel}
	Suppose (Hyp) and fix $z, z_1, z_2 \in \rho(H_0)$. Then the following relations hold true:
	\begin{enumerate}[label=(\roman*)]
		\item $AR_{H_0}(z) \in \mathcal{B}(\H,\H'),
		\quad
		\overline{R_{H_0}(z)B^*} = [B(H_0^*-\overline{z})^{-1}]^* \in \mathcal{B}(\H',\H)$,
		
		\item
			$\overline{R_{H_0}(z_1)B^*}-\overline{R_{H_0}(z_2)B^*}
			= (z_1-z_2) R_{H_0}(z_1) \overline{R_{H_0}(z_2)B^*}
			= (z_1-z_2) R_{H_0}(z_2) \overline{R_{H_0}(z_1)B^*}$,
		
		\item $Q(z) = A\overline{R_{H_0}(z)B^*}$,
		\quad
		$Q(\overline{z})^* = B\overline{R_{H_0}(\overline{z})^* A^*}$,
		
		\item $\ran(\overline{R_{H_0}(z)B^*}) \subseteq \dom(A),
		\quad
		\ran(\overline{R_{H_0}(\overline{z})^* A^*}) \subseteq \dom(B)$,
		
		\item $Q(z_1) -Q(z_2)
			= (z_1-z_2) AR_{H_0}(z_1) \overline{R_{H_0}(z_2)B^*}
			= (z_1-z_2) AR_{H_0}(z_2) \overline{R_{H_0}(z_1)B^*}$.
	\end{enumerate}
\end{lemma}

\begin{proof}
	See Lemma\til2.2 in \cite{GLMZ05}.
\end{proof}

%%%%

We can construct now the extension of the perturbed operator $H_0+B^*A$.

\begin{lemma}[Extension of operators with factorizable potential]
	\label{lem_ext}
	Suppose (Hyp). Fix $z_0 \in\rho(H_0)$, given by Lemma\til\ref{lemBSop}, such that $-1\in\rho(Q(z_0))$. The operator
	\begin{equation}\label{R}
	R_{H}(z_0) = R_{H_0}(z_0) - \overline{R_{H_0}(z_0)B^*} (1+Q(z_0))^{-1} AR_{H_0}(z_0)
	\end{equation}
	defines a densely defined, closed, linear operator $H$ in $\H$, which have $R_{H}(z_0)$ as resolvent,
	\begin{equation*}
	R_{H}(z_0) = (H-z_0)^{-1},
	\end{equation*}
	and which is an extension of $H_0 + B^*A$.
	%
	%Moreover, for $z\in\rho(H_0)$, then $z\in\sigma_p(H)$ if and only if $-1\in\sigma_p(Q(z))$, and $z\in\rho(H)$ if and only if $-1\in\rho(Q(z))$.
	%the subspaces $\ker(H-zI_\H)$ and $\ker(I_{\H'}+Q(z))$ are isomorphic.
\end{lemma}
\begin{proof}
	We refer to Theorem\til2.3 in \cite{GLMZ05}. See also the work by Kato \cite{Kato66}.
\end{proof}

%%%%%%%%%%%

Finally, we enunciate the Birman-Schwinger principle.

\begin{lemma}[Birman-Schwinger principle]
	\label{lemBS}
	Suppose (Hyp). Let $z_0 \in\rho(H_0)$, given by Lemma\til\ref{lemBSop}, such that $-1\in\rho(Q(z_0))$ and $H$ be the extension of $H_0+B^*A$, given by Lemma\til\ref{lem_ext}.
	Fix $z \in \sigma_p(H)$ with eigenfunction $0\neq\psi\in\dom(H)$, i.e. $H\psi=z\psi$. 
	
	Setted $\phi := A\psi$, we have that $\phi\neq0$ and
	\begin{enumerate}[label=(\roman*)]
		\item if $z \in \rho(H_0)$ then
		\begin{equation*}
		Q(z)\phi = -\phi
		\end{equation*}
		and in particular
		\begin{equation*}
		1 \le \n{Q(z)}_{\H'\to\H'} \le \alpha\beta\Lambda(z);
		\end{equation*}
		
		\item 
		if $z\in \sigma(H_0)\setminus\Omega$ and $\psi\in X_{\text{loc}}$, i.e. $\chrt_K \psi \in X$ for every compact set $K \subset \R^n$ where $\chrt_K$ is the indicator function of $K$,	
		then
		\begin{equation}\label{embQ}
		\lim_{\e\to0^\pm} (\varphi,Q(z+i\e)\phi)_{\H'} = -(\varphi, \phi)_{\H'}
		\end{equation}
		for every $\varphi \in \H'$ compactly supported, and in particular
		\begin{equation}\label{embQest}
		1 \le \liminf_{\e\to0^\pm} \n{Q(z+i\e)}_{\H'\to\H'} \le \alpha\beta \lim_{\e\to0^\pm} \Lambda(z+i\e).
		\end{equation}
	\end{enumerate}
\end{lemma}

\begin{proof}
	Let us prove only case (ii) for the embedded eigenvalues, being case (i) similar and easier, adapting the argument of Lemma\til1 in \cite{KK66} 
	%and of Theorem\til3.2 in \cite{GLMZ05} 
	for the non-embedded eigenvalues.
	
	Noted that $H\psi =z \psi$ is equivalent to
	\begin{equation}\label{bs0}
		\psi = (z-z_0) R_{H}(z_0) \psi,
	\end{equation}	
	we obtain from \eqref{R} that
	\begin{equation}\label{bs1}
		(H_0 - z -i\e) R_{H_0}(z_0) \psi
		=
		-(z-z_0) \overline{R_{H_0}(z_0)B^*} (1+Q(z_0))^{-1} AR_{H_0}(z_0) \psi
		-i\e R_{H_0}(z_0) \psi.
	\end{equation}
	Define $\widetilde{\psi}=(1+Q(z_0))^{-1} AR_{H_0}(z_0)\psi$. If $\widetilde{\psi}=0$, by \eqref{bs1} follows $(H_0-z)R_{H_0}(z_0)\psi=0$. Since $0\neq R_{H_0}(z_0)\psi \in \dom(H_0)$, we get $z\in\sigma_p(H_0)$, which however is empty by our assumptions on $H_0$. Thus, we proved $\widetilde{\psi}\neq0$. 
	Moreover, we can show the identity 
	\begin{equation}\label{bs2}
		\phi = A\psi = (z-z_0) (1+Q(z_0))^{-1} AR_{H_0}(z_0)\psi = (z-z_0) \widetilde{\psi},
	\end{equation}
	from which in particular $\phi\neq 0$.
	Indeed, by \eqref{R} and (iii) of Lemma\til\ref{lem-rel}, it follows that
	\begin{equation*}
		AR_{H}(z_0) = (1+Q(z_0))^{-1}AR_{H_0}(z_0)
	\end{equation*}
	which combined with \eqref{bs0} gives us \eqref{bs2}.
	
	Let us prove the limit \eqref{embQ}. Multiplying by $(1+Q(z_0))^{-1} AR_{H_0}(z+i\e)$ both sides of \eqref{bs1}, we obtain
	\begin{equation*}
		\begin{split}
		\widetilde{\psi}
		=
		&-(z-z_0) (1+Q(z_0))^{-1} A R_{H_0}(z+i\e) \overline{R_{H_0}(z_0)B^*} \widetilde{\psi}
		\\
		&- i\e (1+Q(z_0))^{-1} A R_{H_0}(z+i\e) R_{H_0}(z_0) \psi
		\end{split}
	\end{equation*}
	and so, by (v) of Lemma\til\ref{lem-rel} and by the resolvent identity, we have
	\begin{equation*}
	\begin{split}
	\widetilde{\psi}
	=
	&- \frac{z-z_0}{z-z_0+i\e} (1+Q(z_0))^{-1} [Q(z+i\e) - Q(z_0)] \widetilde{\psi}
	\\
	&- \frac{i\e}{z-z_0+i\e} (1+Q(z_0))^{-1} A [R_{H_0}(z+i\e) - R_{H_0}(z_0)] \psi
	\\
	=
	& \, \widetilde{\psi} - \frac{z-z_0}{z-z_0+i\e} (1+Q(z_0))^{-1} (1+Q(z+i\e))\widetilde{\psi}
	\\
	&- \frac{i\e}{z-z_0+i\e} (1+Q(z_0))^{-1} A R_{H_0}(z+i\e)\psi,
	\end{split}
	\end{equation*}
from which, using identity \eqref{bs2}, we finally arrive to
\begin{equation*}
	Q(z+i\e) \phi
	= -\phi
	- i\e A R_{H_0}(z+i\e)\psi.
\end{equation*}

Fixed $\varphi \in \H'$ with compact support $K$, we have that
\begin{equation}\label{Qint_eq}
(\varphi, Q(z+i\e) \phi)_{\H'}
= -(\varphi, \phi)_{\H'}
- i\e (\varphi, A R_{H_0}(z+i\e)\psi)_{\H'}.
\end{equation}
Since
\begin{equation*}
	\begin{split}
	(\varphi, A R_{H_0}(z+i\e)\psi)_{\H'}
	&=
	\int_{K} \varphi^{\dagger} \cdot A R_{H_0}(z+i\e)\psi
	\\
	&=
	\int_{K} R_{H_0}(z-i\e) A^* \varphi^{\dagger} \cdot \psi
	\\
	&=
	(\varphi, A R_{H_0}(z+i\e) [\chrt_K \psi] )_{\H'},
	\end{split}
\end{equation*}
by (Hyp) we get
\begin{equation*}
	|(\varphi, A R_{H_0}(z+i\e)\psi)_{\H'}| 
	\le
	\n{\varphi}_{\H'} \n{A R_{H_0}(z+i\e) [\chrt_K \psi]}_{\H'}
	\le
	\alpha\Lambda(z+i\e) \n{\varphi}_{\H'} \n{\chrt_K \psi}_X,
\end{equation*}
and so the last term in \eqref{Qint_eq} vanishes as $\e\to0^{\pm}$, proving \eqref{embQ}.

Finally, we prove the first inequality in \eqref{embQest}, being the second one given by Lemma\til\ref{lemBSop}. 
For $n\in \N$, let $\chi_n(x) := \chi(x/n)$ where $\chi \in C_0^\infty(\R^n)$ is a cut-off function such that $\chi(x)=1$ for $|x|\le 1$ and $\chi(x)=0$ for $|x|\ge 2$. Since, by \eqref{embQ},
\begin{equation*}
	|(\chi_n \phi , \phi)_{\H'}|
	=
	\lim_{\e\to0^{\pm}}
	|(\chi_n \phi, Q(z+i\e)\phi)_{\H'}| \le
	\n{\chi_n \phi}_{\H'} \n{\phi}_{\H'} 
	\liminf_{\e\to0^{\pm}}
	\n{Q(z+i\e)}_{\H'\to\H'}  
\end{equation*}
we get \eqref{embQest} letting $n\to+\infty$.
\end{proof}

\section{Proof of the Theorems}\label{sec4}

Here we simply apply the results of the last section to our problem. In our case, $\H=\H'=L^2(\R^n;\C^N)$, $X=\bigcap_{j=1}^n L^1_{x_j}L^2_{\widetilde{x}_j}$, $H_0$ is the free Dirac operator $\D_{0}$, and the factorization of $V$ is reached considering the polar decomposition $V=UW$, i.e. $W=(V^* V)^{1/2}$ and the unitary matrix $U$ is a partial isometry, and setting $A=W^{1/2}$ and $B=W^{1/2}U^*$.

It is easy to see that hypothesis (Hyp) holds thanks to Lemma\til\ref{lem2} with
\begin{equation*}
	\alpha=\beta=\n{V}_Y^{1/2},
	\quad
	\Lambda(z)=nC \left[n + \left|\frac{z+m}{z-m}\right|^{\sgn(\Re z)/2}\right],
	\quad
	\Omega =
	\begin{cases}
	\{-m,m\} & \text{if $m\neq 0$,}
	\\
	\varnothing & \text{if $m=0$.}
	\end{cases}
\end{equation*}
Indeed, for $\varphi\in\H$,
	\begin{equation*}
		\begin{split}
		\n{A R_{\D_{0}}(z) \varphi}_{\H} 
		&\le 
		\sum_{j=1}^{n} 
		\n{A \, \chi_j(|z^2-m^2|^{-1/2}D) R_{\D_{0}}(z) \varphi}_{\H} 
		\\
		&\le 
		C \left[n + \left|\frac{z+m}{z-m}\right|^{\sgn(\Re z)/2}\right]
		\sum_{j=1}^{n} 
		\n{A}_{L^2_{x_j} L^\infty_{\widehat{x}_j}}
		\n{\varphi}_{L^1_{x_j} L^2_{\widehat{x}_j}} 
		\\
		&
		\le nC \left[n + \left|\frac{z+m}{z-m}\right|^{\sgn(\Re z)/2}\right]
		\n{V}^{1/2}_{Y} \n{\varphi}_{X},
		\\
		\n{B^* \varphi}_X
		&\le
		\n{V}_Y \n{\varphi}_{\H},
		\end{split}
	\end{equation*}
where we used the relation
\begin{equation*}
\n{A}_{L^2_{x_j} L^\infty_{\widehat{x}_j}}
=
\n{B^*}_{L^2_{x_j} L^\infty_{\widehat{x}_j}}
=
\n{W^{1/2}}_{L^2_{x_j} L^\infty_{\widehat{x}_j}}
=
\n{V}^{1/2}_{L^1_{x_j} L^\infty_{\widehat{x}_j}}.
\end{equation*}
To see that there exists $z_0 \in \rho(\D_0)$ such that $\Lambda(z_0)<(\alpha\beta)^{-1}$, define 
\begin{equation*}
	C_0=[n(n+1)C]^{-1},
	\quad
	\V = [(n+1)C_0/\n{V}_Y - n]^2.
\end{equation*}
Since from the hypothesis of Theorems\til\ref{thm1}\til\&\til\ref{thm2} we have $\n{V}_Y < C_0$ and so $\V>1$, the condition
$1 \le \alpha\beta\Lambda(z)$
is equivalent to $\V \le |z/z|$ if $m=0$, and to
\begin{equation}\label{disks}
\left( \Re z - \sgn(\Re z) m\,\frac{\V^2+1}{\V^2-1}\right)^2 + \Im z^2 
\le
\left(m\,\frac{2\V}{\V^2-1}\right)^2
\end{equation}
if $m>0$. Then, if $m=0$ it is sufficient to choose $z_0\in\C\setminus\R$, whereas if $m>0$ we take $z_0 \in \rho(\D_0)$ outside the disks in the statement of Theorem\til\ref{thm1}. 
Finally, it is trivial to see that if $\psi\in\H$, then $\psi\in X_{\text{loc}}$, since, for every compact set $K\in\R^n$,
we have
$
\n{\chrt_K \psi}_{L^1_{x_j} L^2_{\widehat{x}_j}}
\le 
\kappa \n{\psi}_{L^2}
$,
for some constant $\kappa$ depending on $K$.

Thus, we can apply Lemma\til\ref{lemBS}, which combined with relation \eqref{disks} and with relation $\V \le |z/z|$ proves 
%the claims of 
Theorem\til\ref{thm1} and Theorem\til\ref{thm2} respectively. 
% The final claim of Theorem\til\ref{thm2},
% concerning the absence of residual spectrum for $\D_{V}$,
% is an immediate consequence of the theory developed in
% \cite{BorisovKrejcirik08-a}.

For the final claim in Theorem\til\ref{thm2},
we will follow the argument in \cite{CLT14} to prove that the potential $V \in Y = \bigcap_{j=1}^n L^1_{x_j} L^\infty_{\widehat{x}_j} (\R^n)$ leaves the essential spectrum invariant and that the residual spectrum of $\D_V$ is absent.
To get the invariance of the essential spectrum, it is sufficient to prove that, fixed $z \in \rho(\D_0)$ such that $-1\in \rho(Q(z))$, the operator $A R_{\D_0}(z)$ is an Hilbert-Schmidt operator, and hence compact. Thus, identity \eqref{R} gives us
\begin{equation*}
R_{\D_V}(z) - R_{\D_0}(z)
=
- \overline{R_{\D_0}(z) B^*}
(1+Q(z))^{-1}
A R_{\D_0}(z)
\end{equation*}
from which follows that $R_{\D_V}(z) - R_{\D_0}(z)$ is compact and so, by Theorem 9.2.4 in \cite{EE},
\begin{equation*}
\sigma_e(\D_V) = \sigma_e(\D_0) = 
(-\infty,m] \cup [m,\infty).
\end{equation*}

To see that $A R_{\D_0}(z)$ is an Hilbert-Schmidt operator, we need to prove that its kernel $A(x) \K(z,x-y)$ is in $L^2(\R^n\times\R^n;\C^N)$, where $\K(z,x-y)$ is the kernel of the resolvent $(\D_0-z)^{-1}$.
By the Young inequality
\begin{equation}\label{HSnorm}
\begin{split}
\n{A(\D-z)^{-1}}_{HS}^2
=
\int_{\R^n} \int_{\R^n} |A(x)|^2 |\K(z,x-y)|^2 dx dy 
%\\
%&=
%\int_{\R^n}
%|V| *
%|\K|^2 dy
%\\
\le
\n{V}_{L^p}
\n{\K}_{L^{2q}}^2
\end{split}
\end{equation}
where $1/p + 1/q =2$. Hence we need to find in which Lebesgue space $L^{2q}(\R^n;\C^N)$ the kernel $\K(z,x)$ lies.
For $z\in\rho(-\Delta)=\C\setminus[0,\infty)$, it is well-know (see e.g. \cite{GS}) that the kernel $\K_0(z,x-y)$ of the resolvent operator $(-\Delta-z)^{-1}$ is given by
\begin{equation*}
\K_0(z,x-y)
= 
\frac{1}{(2\pi)^{n/2}} 
\left( \frac{\sqrt{-z}}{|x-y|}\right)^{\frac{n}{2}-1}
K_{\frac{n}{2}-1} (\sqrt{-z}|x-y|)
\end{equation*}
where $K_{\nu}(w)$ is the modified Bessel function of second kind and we consider the principal branch of the complex square root.
Fixed now $z\in\rho(\D_0)=\C\setminus\{\zeta\in\R \colon |\zeta|\ge m\}$, from the identity
\begin{equation*}
(\D_0-zI_N)^{-1} = (\D_0+zI_N)(-\Delta+m^2-z^2)^{-1}I_N
\end{equation*}
and relations 9.6.26 in \cite{AS} for the derivative of the modified Bessel functions, we obtain
\begin{align*}
\K(z,x-y)
&=
\frac{1}{(2\pi)^{n/2}}
\left(\frac{k(z)}{|x-y|}\right)^{\frac{n}{2}}
\vect{\alpha} \cdot (x-y) 
K_{\frac{n}{2}}(k(z) |x-y|) 
\\
& \quad +
\frac{1}{(2\pi)^{n/2}}
\left(\frac{k(z)}{|x-y|}\right)^{\frac{n}{2}-1}
(m\alpha_0+z)
K_{\frac{n}{2}-1} (k(z) |x-y|)
\end{align*}
where for the simplicity
$k(z) = \sqrt{m^2-z^2}$. From the limiting form for the modified Bessel functions
\begin{align*}
K_{\nu}(w) 
&\sim \frac{1}{2} \Gamma(\nu) \left(\frac{w}{2}\right)^{-\nu}
&&\text{for $\Re \nu >0$ and $w \to 0$,}
\\
K_{0}(w) 
&\sim -\ln w
&&\text{for $w \to 0$,}
\\
K_\nu(w)
&\sim \sqrt{\frac{\pi}{2w}} e^{-w}
&&\text{for $z\to\infty$ in $|\arg z| \le 3\pi/2-\delta$,}
\end{align*}
we obtain that
\begin{equation*}
\n{\K(z,x)}
\le C(n,m,z)
\left\{
\begin{aligned}
& \frac{1}{|x|^{n-1}}
&\text{if $|x|\le x_0(n,m,z)$}
\\
& |x|^{-(n-1)/2} e^{-\Re k(z) |x|}
&\text{if $|x|\ge x_0(n,m,z)$}
\end{aligned}
\right.
\end{equation*}
for some positive constants $C(n,m,z)$, $x_0(n,m,z)$ depending on $z$. Hence is clear that $\K(z,x) \in L^{2q}(\R^n; \C^N)$ for $2q<n/(n-1)$ and, consequently, from equation \eqref{HSnorm} we have that $A(\D_0-z)^{-1}$ is an Hilbert-Schmidt operator if $V \in L^p(\R^n;\C^N)$ for $p>n/2$.
Since, by \eqref{YinLn}, $V \in L^n(\R^n;\C^N)$, the proof of the identity $\sigma_e(\D_V)=\sigma_e(\D_0)$ is complete.

%%%%%%%%%%

Finally, to get the absence of residual spectrum, since $\rho(\D_0)=\C\setminus\sigma_e(\D_0)$ is composed by one, or two in the massless case, connected components which intersect $\rho(\D_V)$ in a non-empty set, by Theorem XVII.2.1 in \cite{GGK} we have
$
\sigma(\D_V) \setminus \sigma_e(\D_V) = \sigma_d(\D_V).
$

%%%%%%%%%%%%%%%%%%%%%%%%%%%%%%%
%%%%%%%%%%%%%%%%%%%%%%%%%%
%%%%%%%%%%%%%%%%%%%%%%%%%%%%%%%

\section*{Acknowledgement}
The authors are members of the Gruppo Nazionale per L'Analisi Matematica, la
Probabilit\`a e le loro Applicazioni (GNAMPA) of the Istituto Nazionale di
Alta Matematica (INdAM).
The third author is partially supported by \textit{Progetti per Avvio alla Ricerca di Tipo\til1 -- Sapienza Universit\`a di Roma}.% (No.AR11916B88A0B96B).

%%%%%%%%%%%%%%%%%%%%%%%%%%%%%%%%%%%%%%%%%%%%%%%%%
%%%%%%%%%%%%%%%%%%%%% References %%%%%%%%%%%%%%%%%%%%%
%%%%%%%%%%%%%%%%%%%%%%%%%%%%%%%%%%%%%%%%%%%%%%%%%

%\bibliographystyle{plain}
%\bibliography{bibDFS}{}

%\begin{comment}

\bibliographystyle{plain}

%\end{comment}

\end{document}